\documentclass[oneside, a4paper,11pt,reqno]{amsart}
\usepackage{xcolor}
\usepackage{ulem}
\usepackage{graphicx}

\newtheorem{hypo1}{Hypothesis}[section]

\newtheorem{coro1}[hypo1]{Corollary}

\newcommand{\Z}{\mathbb{Z}}

\newtheorem{hypo}{Hypothesis}

\newtheorem{thm}[hypo]{Theorem}

\newtheorem{lem}[hypo]{Lemma}

\def\PP{\mathbb{P}}
\def\RR{\mathbb{R}}
\def\ZZ{\mathbb{Z}}

\def\EE{\mathbb{E}}
\def\NN{\mathbb{N}}

\let\BFseries\bfseries\def\bfseries{\BFseries\mathversion{bold}} % formulas in headings bold

\def\dd{\mbox{d}}

\allowdisplaybreaks

\begin{document}

\title{Random walks and branching processes in correlated Gaussian environment}

 \author{ F. Aurzada}
\address{AG Stochastik, Fachbereich Mathematik, Technische Universit\"at Darmstadt, Schlossgartenstr.\ 7, 64289 Darmstadt, Germany.}
\email{aurzada@mathematik.tu-darmstadt.de}

\author{A. Devulder}
\address{
Laboratoire de Math\'ematiques de Versailles, UVSQ, CNRS, Universit\'e Paris-Saclay, 78035 Versailles, France
}
\email{devulder@math.uvsq.fr }

\author{N. Guillotin-Plantard}
\address{Institut Camille Jordan, CNRS UMR 5208, Universit\'e de Lyon, Universit\'e Lyon 1, 43, Boulevard du 11 novembre 1918, 69622 Villeurbanne, France.}
\email{nadine.guillotin@univ-lyon1.fr}

\author{F. P\`ene}
\address{Universit\'e de Brest and IUF,
LMBA, UMR CNRS 6205, 29238 Brest cedex, France}
\email{francoise.pene@univ-brest.fr}


\date{\today}

\begin{abstract}
We study persistence probabilities for random walks in correlated Gaussian random environment first studied by Oshanin, Rosso and Schehr  \cite{oshaninetal}. From the persistence results, we can deduce properties of critical branching processes with offspring sizes geometrically distributed with correlated random parameters. More precisely,
we obtain estimates on the tail distribution of its total population size, of its maximum population, and of its extinction time.
\end{abstract}

\maketitle

\section{Introduction and main results}
\subsection{Random walks in correlated random environment}
{\it Random walks in random environment}
(RWRE, for short) model the displacement of a particle in an inhomogeneous medium.
We consider a nearest-neighbor random walk $S=(S_n)_{n\ge 0}$, in $\mathbb Z$, in a random environment.
Let $\omega:=(\omega_i)_{i\in\mathbb Z}$ be a stationary sequence of random variables with values in $(0,1)$ defined on the probability space
$(\Omega,\mathcal{F}, P)$.
A realization of $\omega$ is called an {\it environment}.
The RWRE $S$ is then defined as follows.
Given $\omega$, under the {\it quenched law} $P_\omega^x$ for $x\in\ZZ$,
$S:=(S_n)_{n\geq 0}$ is a Markov chain satisfying $P_\omega^x[S_0=x]=1$ and
%\begin{equation}\label{eqDefRWFRE}
%    \forall i\in\mathbb Z,\ \forall n\in\mathbb N,
%\quad
%    P_{\omega}^x[S_{n+1}=i+1 | S_n = i]=\omega_i=1- P_{\omega}^x[S_{n+1}=i-1 | S_n = i].
%\end{equation}
for every $n\in\mathbb N$, $k\in\Z$ and $i\in\Z$,
\begin{equation}\label{eqDefRWFRE}
    P_{\omega}^x[S_{n+1}=k|S_n=i]=\left\{
\begin{array}{ll}
\omega_i & \text{ if } k=i+1,\\
1-\omega_i & \text{ if } k=i-1,\\
0      & \text{ otherwise}.
\end{array}
\right.
\end{equation}
We simply write $P_\omega$ for $P_\omega^0$.
We also define  the {\it annealed law} by
$$
    \PP[\cdot]
:=
    \int P_{\omega} [\cdot]\,  {\rm d} P(\omega).
$$
The expectations with respect to $\PP$, $P_{\omega}$, and $P$ will be denoted by $\EE$, $E_{\omega}$, and $E$, respectively.
This model has many applications in physics and in biology, see e.g. Hughes \cite{Hug} and Oshanin, Rosso and Schehr  \cite{oshaninetal}.

The case when $(\omega_i)_i$ is a sequence of independent identically distributed
random variables
%(that is $r=\delta_0$)
has been widely studied since the seminal works by Solomon \cite{S75}, who proved a recurrence and transience criterium and a law of large numbers,
and by Sinai \cite{S82}, who proved a localization result in the recurrent case with some additional assumptions. We refer e.g.\ to R{\'e}v{\'e}sz \cite{Revesz}, Zeitouni \cite{Zeitouni}, and Shi \cite{Shi} and to the references therein for more properties of RWRE in such environments.

In the present paper  we consider a correlated context that has been recently introduced in statistical physics (see Oshanin, Rosso and Schehr  \cite{oshaninetal}). Before defining our setup more precisely, we introduce some more notation.
In the study of RWRE, the {\it potential} $V=(V(k))_{k\in\mathbb Z}$ plays a major role
(see for example formulae \eqref{eqProbaAtteinte} and \eqref{InegEsperanceZeitouni} below).
It is defined as follows:
$$
    X_i:=\log(1-\omega_i)/\omega_i,
\qquad
    V(0):=0,
\qquad
    V(k+1):=V(k)+X_{k+1}
$$
for every $i\in\mathbb Z$ and $k\in\mathbb Z$.
%We set $X_i:=\log(1-\omega_i)/\omega_i$, $i\in\mathbb Z$,  and recall
%that the potential $V=(V(k))_{k\in\mathbb Z}$ is defined by $V(0):=0$ and
%$V(k+1):=V(k)+X_{k+1}$ for every $k\in\mathbb Z$.
To ensure that no side (left/right) is privileged, one has to assume that $E[X_i]=0 $.
We reinforce this by assuming that the $X_i$'s are standard Gaussian random variables.
We say that $S$ is a {\it random walk in a correlated Gaussian environment} (RWCGE).
It is worth noting that $(S_n)_{n\in\NN}$ is not Markovian under $\PP$.
We set $r(j):=E[X_0X_j]= E[X_kX_{k+j}]$, $j\in\Z$, $k\in\Z$. Note that for $n\in\mathbb N$, the variance
$\sigma_n^2$ of $V(n)$ is given by $\sigma^2_n=\sum_{i,j=1}^nr(i-j)$.

Our setup is the following. Let $H\in[\frac 12,1)$.
We assume that $(r(n))_{n\in\NN}$ is non-negative and $(2H-2)$-regularly varying
(i.e. $(n^{2-2H}r(n))_{n\in\NN}$ is slowly varying).
This ensures that
\begin{equation}\label{formulaVarVn}
    \sigma^2_n
:=
    \text{Var}(V(n))=n^{2H}\ell(n),
\qquad
    n\in\NN,
\end{equation}
for some function $\ell:[0,+\infty[\rightarrow[0,+\infty[$, slowly varying at infinity
(see for example Bingham, Goldie, Teugels \cite[Prop 1.5.8, Prop 1.5.9a]{BGT} or Taqqu \cite[Lemma 3.1]{taqqu}).
A sequence satisfying \eqref{formulaVarVn}
is said to have {\it long range dependence} if $H>1/2$, see \cite{samorodnitsky} for a recent overview.
Due to (Taqqu \cite{taqqu}, Lemma 5.1), the process
$((V(\lfloor nt\rfloor)/\sigma_n)_{t\in\RR})_n$ converges in distribution as $n\to+\infty$
to a two-sided fractional Brownian motion $B_H:=(B_{H}(t))_{t\in \mathbb{R}}$ with Hurst parameter $H$.
Recall that $B_H$ is a centered real Gaussian process such that $B_{H}(0)\equiv 0$ with covariance function given by
$$
    E\big[ B_{H}(t) B_{H}(s)\big]
=
    \frac{1}{2} \left(|t|^{2H} + |s|^{2H} - |t-s|^{2H}\right)
\qquad
    s\in {\mathbb R},t\in\mathbb{R}.
$$
This process $(B_H(t),\ t\in\RR)$ has stationary increments and is self-similar of index $H$, that is,
$(B_H(c t) )_{t\in\RR}$   and $\big(c^H B_H(t)\big)_{t\in\RR}$ have the same law for every $c>0$.
\\
Very few results are known in our context (see \cite{oshaninetal} and \cite{AurzadaBaumgartenjpa}).
In the special case when the potential is itself a fractional Brownian motion of
Hurst exponent $H\in(1/2,1)$,
Kawazu, Tamura, Tanaka \cite{KTT} (see their Theorem 5) and Schumacher \cite{Schu} proved the weak convergence of $\left({S_n}/{(\log n)^{1/H}}\right)_{n\geq 2}$
to a non-degenerate law.
\\
We define the first hitting time $\tau(k)$ of $k\in \ZZ$ by the random walk $S$, that is,
$$
    \tau(k)
:=
    \inf\{n\geq 1,\ S_n=k\},
\qquad
    k\in\mathbb{Z}.
$$
In this paper we are concerned with the  persistence probability of $S$, i.e. the annealed probability that the RWCGE $S$ does not visit the site $-1$ before time $N$.
We refer to Aurzada and Simon \cite{AS} for a recent survey about persistence from a mathematical point of view. We will use the recent results of \cite{AGPP} and
the new approach used therein.

Persistence has also received a considerable attention in statistical physics, see e.g.\  Bray, Majumdar and Schehr \cite{BMS13} and Majumdar \cite{Maj1}. Persistence is perceived as a measure of how quickly a physical system started in a disordered state returns to the equilibrium.

Our first main result is the following.

\medskip

\begin{thm}\label{persist}
Let $H\in[\frac 12,1)$.
Assume that $(X_i)_{i\in\mathbb Z}$ is a stationary sequence of standard Gaussian random variables.
Assume that $(r(n))_n$ is non-negative and $(2H-2)$-regularly varying.
Then there exist $N_0\in\mathbb N$ and a slowly varying function at infinity $L_0$  such that, for every $N\ge N_0$,
$$
    \frac{(\log N)^{-\left(\frac{1-H}{H}\right)}}{L_0(\log N)}
\leq
    \PP\Big[ \min_{k=1,\ldots, N} S_k > -1 \Big] =\PP[\tau(-1)>N]\leq  (\log N)^{-\left(\frac{1-H}{H}\right)} L_0(\log N).
$$
Moreover if $V=B_H$, then there exist $c>0$ and $N_0\in\mathbb N$ such that, for every $N\ge N_0$,
$$
    (\log N)^{-\left(\frac{1-H}{H}\right)} e^{-c\sqrt{\log \log N}} \leq
    \PP\Big[ \min_{k=1,\ldots, N} S_k > -1 \Big] \leq  (\log N)^{-\left(\frac{1-H}{H}\right)} (\log \log N)^{c}.
$$
\end{thm}
\medskip

%Corresponding results for RWRE when the environment is a sequence of independent and identically distributed random variables

\subsection{Branching Processes in Random Environment}
The second object of study in this paper is Branching Processes in Random Environment.
They are an important generalization of the Galton Watson process, where the reproduction law
depends on a random environment indexed by time.
This model was first introduced by Smith and Wilkinson \cite{SmithWilkinson}.
In a few papers, the reproduction laws are assumed to be stationary and ergodic,
we refer to Athreya and  Karlin (\cite{AthreyaKarlin1} and \cite{AthreyaKarlin2})
for basic results in this general case.
However in most studies, the reproduction laws are supposed to be independent and identically distributed,
and they are often assumed to be geometrical laws.
See e.g. Grama, Liu and Miqueu \cite{GramaLiuMiqueu} for a recent overview and bibliography on the subject.

It is natural to consider cases for which the reproduction laws of
the different generations are correlated.
To this aim, we use the well known correspondence between recurrent random walks in random environment and critical branching processes in random environment with geometric distribution of offspring sizes (see e.g.\ Afanasyev \cite{Afa}).
%Let $\tau(-1)$ be the first hitting time of $-1$ by the random walk $S$.
We consider the process $(Z_n)_{n\in\NN}$ defined by
\begin{equation}\label{eqDefBPRE}
    Z_0
:=1
\qquad \mbox{and}\qquad
    Z_n
:=
    \sum_{k=1}^{\tau(-1)}{\bf 1}_{\{S_{k-1}=n-1, S_k=n\}},
\qquad
    n\geq 1.
\end{equation}
In other words, $Z_n$ is, for $n\geq 1$, the number of steps from $n-1$ to $n$  made by the RWCGE $S$ before reaching negative values.
This process $(Z_n)_{n\in\mathbb{N}}$ is a {\it Branching Process in a Correlated Gaussian Environment (BPCGE)}.

More precisely, let $O_{n,k}$ be the number of steps $(n\rightarrow n+1)$
between the $k$-th and the $(k+1)$-th step $(n-1\rightarrow n)$ for $(n,k)\in \NN\times \NN^*\setminus\{(0,1)\}$,
and between $0$ and $\tau(-1)$ for $n=0$ and $k=1$, where $\NN^*:=\NN\setminus\{0\}$.
Observe that, given $\omega$, $(O_{n,k})_{n\geq 0, k\geq 1}$
is a double sequence of independent random variables and that
\begin{equation}\label{eqDefRWFREBis}
    Z_0=1,
\quad
    Z_{n+1}
:=
    \sum_{k=1}^{Z_n}O_{n,k},
\quad
    n\geq 0,
\qquad
    P_\omega(O_{n,k}=N)=(1-\omega_n)\omega_n^N,
\quad   (k,n,N)\in\NN^*\times \NN^2.
\end{equation}
Hence, the number of offsprings $O_{n,k}$ of the $k$-th particle of generation $n$ (of the BPCGE $Z$)
is, under $P_\omega$, a geometric random variable on $\mathbb{N}$ with mean $e^{-X_n}$.
So the BPCGE is critical, and in particular there is almost surely extinction of this BPCGE
(see e.g. Tanny \cite{Tanny}, eq. (2) and the terminology before, coming from Tanny \cite{Tanny2}, Thm 5.5).
Note that $\tau(-1)=2\sum_{j=0}^\infty Z_j-1$.
Thus, the total population size $\sum_{j=0}^\infty Z_j$ of the BPCGE $(Z_n)_n$ satisfies
$
    \PP\big(\sum_{j=0}^\infty Z_j>N\big)
=
    \PP[\tau(-1)>2N-1]
=
    \PP(\min_{k=1,\dots, 2N-1}S_k>-1)
$,
$N\in\NN^*$. Consequently, Theorem \ref{persist} leads to the following result.

\medskip

\begin{coro1}[{\bf Total population size of BPCGE}] \label{CorolBPRETotalSize}
Under assumptions of Theorem \ref{persist},
there exist $N_0\in\mathbb N$ and a slowly varying function at infinity  $L_0$ such that  the total population size of the BPCGE $Z$ before its extinction satisfies, for every $N\ge N_0$,
$$
    \frac{(\log N)^{-\left(\frac{1-H}{H}\right)}}{L_0(\log N)}
\leq
       \PP\bigg[ \sum_{j=0}^\infty Z_j>N \bigg]
 \leq  (\log N)^{-\left(\frac{1-H}{H}\right)} L_0(\log N).
$$
Moreover if $V=B_H$, then there exist $c>0$ and $N_0\in\mathbb N$ such that, for every $N\ge N_0$,
$$
  (\log N)^{-\left(\frac{1-H}{H}\right)} e^{-c\sqrt{\log \log N}}
\leq
       \PP\bigg[ \sum_{j=0}^\infty Z_j>N \bigg]
\leq
    (\log N)^{-\left(\frac{1-H}{H}\right)} (\log \log N)^{c}.$$
\end{coro1}

Let $\mathcal T:=\inf\{n\geq 1; Z_n=0\}$ be the extinction time of the BPCGE $Z$.

Our second main result deals with the survival probability $\PP[ \mathcal T > N ]$ of BPCGE.

\medskip

\begin{thm}[{\bf Extinction time of BPCGE}] \label{main1}
Under assumptions of Theorem \ref{persist},
there exist $c>0$, $C>0$ and $N_0\in \mathbb N$ such that,
for every $N\ge N_0$,
\begin{equation}\label{cinq}
    N^{-(1-H)}\sqrt{\ell(N)} (\log N)^{-c}
\leq
    \PP[ \mathcal T > N ]
\leq
    C N^{-(1-H)}\sqrt{\ell(N)}.
\end{equation}
\end{thm}

An easy consequence of the previous results is the following estimate on the maximum population size
$\sup_{j\geq 0} Z_j$
of the BPCGE $Z$
before its extinction.

\medskip

\begin{coro1}[{\bf Maximum population size of BPCGE}]
Under assumptions of Theorem \ref{persist},
there exist $N_0\in\mathbb N$ and a slowly varying function at infinity  $L_0$  such that
the maximum population size of the BPCGE $Z$ before its extinction satisfies, for every $N\ge N_0$,
$$
    \frac{(\log N)^{-\left(\frac{1-H}{H}\right)}}{L_0(\log N)}
\leq
    \PP\bigg[ \sup_{j\ge 0} Z_j>N \bigg]
 \leq  (\log N)^{-\left(\frac{1-H}{H}\right)} L_0(\log N).
$$
Moreover if $V=B_H$, there exist $c>0$ and $N_0\in\mathbb N$ such that, for every $N\ge N_0$,
$$
  (\log N)^{-\left(\frac{1-H}{H}\right)} e^{-c\sqrt{\log \log N}}
\leq
    \PP\bigg[ \sup_{j\ge 0} Z_j>N \bigg]
\leq
    (\log N)^{-\left(\frac{1-H}{H}\right)} (\log \log N)^{c}.$$
\end{coro1}
\begin{proof}
As in (\cite{Afa} eq.\ (42)), we note that $\sup_{j\geq 0} Z_j\le \sum_{j\geq 0} Z_j\le \mathcal T
\sup_{j\geq 0} Z_j$.
Consequently,
$
    \mathbb P\big[\sup_{j\geq 0} Z_j>N\big]
\le
    \mathbb P\big[\sum_{j\geq 0}Z_j>N\big]
$
and the upper bound follows from the upper bound of Corollary \ref{CorolBPRETotalSize}.

Moreover
$
    \mathbb P\big[\sup_{j\geq 0}Z_j>N\big]
\ge
    \mathbb P\big[\sum_{j\geq 0} Z_j>{N}^2]-\mathbb P[\mathcal T>{N}\big]
$.
This, Corollary \ref{CorolBPRETotalSize} and \eqref{cinq} lead to the lower bound.
\end {proof}
\noindent{\bf Remark.} The proofs of the upper bounds in the above results remain true if our regular variation assumption
on $r$ fails provided \eqref{formulaVarVn} holds.
For the lower bound, we can replace our regular variation assumption on $r$ by \eqref{formulaVarVn} and $m^2\, r(m)=O(\sigma^2_m)$ (which holds true
in particular if $r$ is decreasing and satisfies
\eqref{formulaVarVn}).

\noindent{\bf Remark.} In Afanasyev \cite{Afa} (see also Vatutin \cite{{Vatutin}} for the stable case) the case of i.i.d.\ environment (where $H=\frac 1 2$) was treated. The above-mentioned correspondence between the random walk in random environment and a branching process in random environment with geometric distributions is used in \cite{Afa} by Afanasyev to deduce the tail of the first hitting time $\tau(-1)$ of $-1$ by the random walk in random environment.  Afanasyev's method is quite efficient since he obtains
$$
    P\big[ \tau(-1)>N \big]
\sim_{N\to+\infty}
    \frac{c}{ \log N}
$$
for some positive constant $c$. However his proof rests on a functional limit theorem for the branching process in random environment
which seems difficult to establish when the reproduction laws of
the different generations are correlated.\\*

We recall the following estimates, that will be useful in the present work.
We shall use the following hitting time formula: If $p<q<r$, then from formula (2.1.4), p.\ 196 in Zeitouni \cite{Zeitouni},
\begin{equation}\label{eqProbaAtteinte}
    P_\omega^q[\tau(r)<\tau(p)]
=
    \Big(\sum_{k=p}^{q-1}e^{V(k)}\Big)\Big(\sum_{k=p}^{r-1}e^{V(k)}\Big)^{-1}.
\end{equation}
Moreover if $g<h<i$, we have (see e.g.\ Lemma 2.2. in Devulder \cite{devulder}  coming from \cite{Zeitouni} p.\ 250)
\begin{equation}\label{InegEsperanceZeitouni}
    E_\omega^{h} [\tau(g)    \wedge\tau(i)]
\leq
    \sum_{k={h}}^{i-1}\sum_{\ell=g}^{k}\Big[\big(1+e^{X_\ell}\big)\exp[V(k)-V(\ell)]\Big].
\end{equation}

The rest of this paper is organized as follows. The upper and lower bounds of Theorem
\ref{persist} are proved in Sections \ref{upperMain} and \ref{lowerMain}, respectively. Section \ref{proofThm2} contains a useful lemma that may be of independent interest and the proof of Theorem \ref{main1}.
\section{Proof of the upper bound in Theorem \ref{persist}}\label{upperMain}
Let $ T(x)$ be the first passage time of the potential $(V(k))_{k\in\NN}$ above/below the level $x\neq 0$.
More precisely, let
$$
    T(x)
:=
    \left\{\begin{array}{ccc}
    \inf\{k\in \NN;\  V(k) \geq x\} & \mbox{if} & x>0,\\*
    \inf\{k\in \NN;\  V(k) \leq x\} & \mbox{if}  & x<0.
\end{array}
\right.
$$
\subsection{First passage times by discrete FBM}
We start by stating a result in the particular case when $V=B_H$.
We set
$$
    \widetilde T(x)
:=
    \left\{\begin{array}{ccc}
    \inf\{k\in\NN;\ B_H(k) \geq x\} & \mbox{if} & x>0,\\*
    \inf\{k\in\NN;\ B_H(k) \leq x\} & \mbox{if}  & x<0.
\end{array}
\right.
$$

In the following theorem, we estimate the probability that the discrete FBM $(B_H(k))_{k\in\NN}$ hits $-x$ before $y$,
for $y$ and large $x$  satisfying some technical conditions.

\medskip
\begin{thm} \label{frank}
Recall that $H\in[\frac 12,1)$.
Let $\alpha>1$.
There exist $c=c(\alpha)>0$ and $x_\alpha>0$ such that for any $y> e$ and any $x>\max(y,x_\alpha)$
such that $\log x \leq [\log (x/y)]^\alpha$, we have
$$
    (x/y)^{-(1-H)/H} \big[\log (x/y) \big]^{-c}
\leq
    P\left[\widetilde T(-x)<\widetilde T(y)\right]
\leq
    c (x/y)^{-(1-H)/H}\big[\log x \big]^{c}.
$$
\end{thm}

It is well known that more precise results can be obtained with martingale techniques when $H=1/2$, however these methods fail when $H\neq 1/2$.

\medskip

\begin{proof}
We fix $\alpha>1$. Throughout the proof we consider only $x>y>e$ such that $\log x \leq [\log (x/y)]^\alpha$.

To see the upper bound, define $b=b(x)=\big\lfloor x^{1/H} (\log x)^{-q/(2H)}\big\rfloor$ with $q>1$,
where for $u\in{\mathbb R}$, $\lfloor u\rfloor$ denotes the integer part of $u$, and $\lceil u\rceil=\lfloor u\rfloor+1$. Then
\begin{eqnarray}
    P\big[\widetilde T(-x)<\widetilde T(y)\big]
&\leq &
    P\big[\widetilde T(y) > b\big]
    +
    P\big[\widetilde T(y)\leq b,\, \widetilde T(-x)<\widetilde T(y)\big]
\notag \\
&\leq &
    P\Big[\max_{k=1,2,\dots,b} B_H(k) < y\Big] + P[\widetilde T(-x)<b]. \label{eqn:firstdist1}
\end{eqnarray}
The first term will give the leading order while the second is of lower order. Let us treat the first term:
\begin{eqnarray}
%    P\big[\widetilde T(y)>b(x)\big]
%=
    P\Big[\max_{k=1,2,\dots,b} B_H(k) < y\Big]
&\leq &
    P\Big[\max_{k=1,2,\dots,b} B_H(k) < \lceil y^{1/H}\rceil^{H}\Big]
\nonumber\\
&= &
    P\Big[\max_{k=1,2,\dots,b} \lceil y^{1/H}\rceil^{-H} B_H(k) < 1\Big]
\nonumber\\
&= &
    P\Big[\max_{k=1,2,\dots,b} B_H\big(k/\big\lceil y^{1/H}\big\rceil\big) < 1\Big]
\nonumber\\
&\leq &
    P\Big[\max_{\ell=1,2,\dots,\lfloor b/\lceil y^{1/H}\rceil \rfloor} B_H(\ell) < 1\Big]
\nonumber\\
&\leq &
    c \left\lfloor b/\big\lceil y^{1/H}\big\rceil \right\rfloor^{-(1-H)}
   % \big(\log \big\lfloor b/\big\lceil y^{1/H}\big\rceil \big\rfloor\big)^{c}
\label{eqLastButOne}\\
&\leq &
    c' ( x/y)^{-(1-H)/H} (\log x)^{\frac{(1-H)q}{2H}}, % [\log (x/y)]^{c'}(\log x)^{c'}
%\nonumber\\
%&\leq &
%    ( x/y)^{-(1-H)/H} %  [\log (x/y)]^{c''},
\label{InegProbaMaxBHInferieuraY}
\end{eqnarray}
for some constants $c>0$ and $c'>0$, 
where estimate \eqref{eqLastButOne} comes from Theorem 11 in \cite{AGPP} having used $H\geq 1/2$,
since $b/\big\lceil y^{1/H}\big\rceil$ is large enough when $x$ is large enough under our hypotheses.
To see that the second term in (\ref{eqn:firstdist1}) is of lower order, notice that
\begin{eqnarray*}
    P\big[\widetilde T(-x)<b\big]
& = &
    P\Big[\min_{k=1,\ldots,b-1} B_H(k) \leq -x\Big]
\\
& \leq &
    P\Big[\min_{s \in[0, b]} B_H(s) \leq -x\Big]
=
    P\Big[\max_{s \in[0, b]} B_H(s) \geq x\Big]
\leq
    2 P\big[B_H(b)\geq x\big],
\end{eqnarray*}
where we used Proposition 2.2. in \cite{KL} in the last inequality since $E[X_kX_{j+k}]\geq 0$ for every $j\in\Z$ and $k\in\Z$.
Consequently if $x$ is large enough,
\begin{eqnarray}
    P\big[\widetilde T(-x)<b\big]
& \leq &
    2P\big[b^H B_H(1)\geq x\big]
\leq
    2\exp\big[-x^2/\big(2b^{2H}\big)\big]
\label{InegProbaTxB}
\\
& \leq &
    \exp[-(\log x)^q/4]
\leq
    \exp[-(\log (x/y))^q/4].
\nonumber
\end{eqnarray}
This together with \eqref{eqn:firstdist1} and \eqref{InegProbaMaxBHInferieuraY} ends the proof of the upper bound in the theorem
since $q>1$.

For the lower bound, define $d=d(x)=\big\lfloor x^{1/H} (\log x)^{q}\big\rfloor$ with $q>1$. Note that
\begin{equation}
    P\big[\widetilde T(-x)<\widetilde T(y)\big]
\geq
    P\big[\widetilde T(-x)\leq d, d<\widetilde T(y)\big]
=
    P\big[\widetilde T(y) > d\big] - P\big[ \widetilde T(y)> d, \widetilde T(-x)>d\big].
\label{eqn:seconddist1}
\end{equation}
The first term in the right hand side of (\ref{eqn:seconddist1}) can be treated as follows.
For large $x$,
\begin{eqnarray}
    P\big[\widetilde T(y) > d\big]
& = &
    P\Big[\max_{k=1,\ldots,d} B_H(k) < y\Big]
\nonumber\\
&\geq &
    P\Big[\sup_{t\in[0,d]} B_H(t) < y\Big]
\nonumber\\
&= &
    P\Big[\sup_{t\in[0,d]}  B_H\big(t/y^{1/H}\big) < 1\Big]
\nonumber\\
&= &
    P\Big[\sup_{s\in[0,d/y^{1/H}]}  B_H(s) < 1\Big]
\nonumber\\
&\geq &
    c \big(d/y^{1/H}\big)^{-(1-H)} \big[\log \big(d/y^{1/H}\big)\big]^{-\frac 1{2H}}
\nonumber\\
&\geq &
  c'  (x/y)^{-(1-H)/H} [\log (x/y)]^{-\frac 1{2H}-\alpha q(1-H)},
\label{InegProbaTildeTy}
\end{eqnarray}
for some constants $c>0$ and $c'>0$, where the last but one estimate comes from Theorem~1 in \cite{AGPP} and is valid for any $H\in(0,1)$,
since $d/y^{1/H}$ is large for large $x$.

It remains to be seen that the second term in the right hand side of (\ref{eqn:seconddist1}) is of lower order. First, note that (using $x\geq y$),
we get
\begin{eqnarray}
&&
    P\big[ \widetilde T(y)> d, \widetilde T(-x)>d\big]
\nonumber\\
& \leq &
    P\Big[ \sup_{k=1,\ldots,d} |B_H(k)|\leq x \Big]
\label{ProbaTyT-x1}
\\
& \leq &
    P\Big[\sup_{t\in[0,d]} |B_H(t)|\leq 2x\Big]
    +P\Big[\exists k\in\{0,\ldots d-1\} : \sup_{t\in[0,1]} \big|B_H(k+t)-B_H(k)\big| > x\Big]
\nonumber\\
& \leq &
    P\Big[\sup_{t\in[0,d]} |B_H(t)|\leq 2x\Big]+d\, P\Big[ \sup_{t\in[0,1]} |B_H(t)|>x\Big].
\label{InegProbaMaxBH1d}
\end{eqnarray}
The second term of the previous line is of lower order ($\leq d e^{-c x^2}$ for large $x$),
by standard large deviation estimates for Gaussian processes (e.g.\ Theorem~12.1 p.\ 139 in Lifshits \cite{lifshits}).
The first term is a small deviation probability (observe that $x/d^H\to 0$ as $x\to+\infty$) and can be treated as follows.
There exists $c>0$ and $c'>0$ such that for large $x$,
$$
    P\Big[\sup_{t\in[0,d]} |B_H(t)|\leq 2x\Big]
=
    P\Big[\sup_{s\in[0,1]} |B_H(s)|\leq 2x/d^H\Big]
\leq
    e^{-c (2x/d^H)^{-1/H}}
\leq
    e^{-c'(\log x)^q},
$$
by small deviation results for FBM (see e.g.\ Li and Shao \cite{lishao}, Theorem 4.6 in Section 4.3).
Thus, \eqref{InegProbaMaxBH1d} gives for large $x$,
$$
    P\Big[ \sup_{k=1,\ldots,d} |B_H(k)|\leq x \Big]
\leq
    2e^{-c'(\log x)^q}
\leq
    2e^{-c'(\log (x/y))^q},
$$
which is negligible compared to the right hand side of \eqref{InegProbaTildeTy} since $q>1$.
This, together with \eqref{eqn:seconddist1}, \eqref{InegProbaTildeTy} and \eqref{ProbaTyT-x1}, proves the lower bound.
\end{proof}
\subsection{First passage times of a potential}
We now state the following result for $V$ (less general than Theorem~\ref{frank}, but sufficient for our purposes).

\medskip

\begin{lem} \label{PERSIbis}
Let $a>0$, $\varepsilon\in(0,1)$ and $q>1$. Let $b_N:=\sup\{k\, :\, \sigma_k\le (\log N)(\log\log N)^{-\frac q{2}}\}$ and let $L$ be the slowly varying function at infinity such that
$b_N=(\log N)^{\frac 1 H}L(\log N)$. Then
$$
    P\left[ T(a\log\log N)\leq b_N \leq T\left(-\frac{(1-\varepsilon)\log N}{2}\right)\right]
\ge 1-\frac{\sigma_{b_N}}{b_N}(\log\log N)^{c},
$$
for some $c>0$ and for all $N$ large enough.
\end{lem}

\medskip

\begin{proof}
We shall show the following two estimates which yield the claim:

First,
\begin{equation}
\label{eqLemmaFrankbis1r}
    \exists \kappa>0,\quad
    P\left[ T\left(-\frac{(1-\varepsilon)\log N}{2}\right) < b_N\right]
\le
    O\left(e^{-\kappa(\log\log N)^{q}}\right)
\end{equation}
as $N\to+\infty$.

Second, there exists $c>0$ such that
for $N$ large enough,
\begin{equation}
\label{eqLemmaFrankbis2r}
    P\big[ T(a\log\log N)> b_N\big]
\le
    \frac{\sigma_{b_N}}{b_N}(\log\log N)^{c}.
\end{equation}

%Set $b_N:=(\log N)^{\frac 1H}(\ell((\log N)^{1/H}))(\log\log N)^{-\frac q{2H}}$.
%Since $\sigma^2_t+\sigma^2_s\le\sigma^2_{t+s}$, due to Khoshnevisan et al. \cite[Prop 2.2]{KL}
Due to \cite[Prop 2.2]{KL} (Remark that the proof of this result holds when only a finite number of random variables is
considered), we have for $N$ large enough,
\begin{eqnarray*}
    P\left[ T\left(-\frac{(1-\varepsilon) \log N}{2}\right) < b_N\right]
& \leq &
    P\left[\max_{k=0,...,b_N}(-V(k))\geq (1-\varepsilon)(\log N)/2\right]
\\
& \leq &
    2P\left[-V(b_N)\geq (1-\varepsilon)(\log N)/2\right]
\\
& = &
    2P\left[\sigma_{b_N}V(1) \geq (1-\varepsilon)(\log N)/2\right]
\\
& \le &
    2\exp\left(-\frac 1{8\sigma^2_{b_N}}(1-\varepsilon)^2(\log N)^2\right)
\\
&\le&
    2\exp\left(-\frac 1{8}(1-\varepsilon)^2(\log\log N)^{q }\right).
\end{eqnarray*}
This gives \eqref{eqLemmaFrankbis1r}.
For  \eqref{eqLemmaFrankbis2r}, note that
\begin{equation}\label{EqProbaTaloglogN}
    P\left[ T(a\log\log N)> b_N\right]
=
    P\left[\max_{k=0,...,b_N}V(k)<a\log\log N\right].
\end{equation}
Let
$c_N$ be so that $\sigma_{c_N}\sim a\log\log N$ as $N\to+\infty$.
Then
\begin{equation}\label{persiV1}
\lim_{N\rightarrow +\infty} P\left[V(c_N)<-a\log\log N\right]
= P\left[B_H(1)<-1\right]>0.
\end{equation}
Due to Slepian lemma,
%(see e.g. Lifshits \cite{lifshits}, Thm. 3 p. 189),
the $X_i$'s are positively associated and so
\begin{multline}\label{persiV2}
P\left[\max_{k=1,...,c_N}V(k)\le 0\right]P\left[V(c_N)<-a\log\log N\right]P\left[\max_{k=0,...,b_N}V(k)<a\log\log N\right]\\
\le P\left[\max_{k=1,...,c_N}V(k)\le 0,\ V(c_N)<-a\log\log N,\ \max_{k=1,...,b_N}(V(c_N+k)-V(c_N))<a\log\log N\right]\\
\le P\left[\max_{k=1,...,b_N+c_N}V(k)\le 0\right].
\end{multline}
But, due to Theorem 11 in \cite{AGPP},
\begin{equation}\label{persiV3}
\limsup_{n\rightarrow +\infty}
   \frac{b_N}{\sigma_{b_N}}P\left[\max_{k=1,...,b_N+c_N}V(k)\le 0\right]=
\limsup_{n\rightarrow +\infty}
   \frac{b_N+c_N}{\sigma_{b_N+c_N}}P\left[\max_{k=1,...,b_N+c_N}V(k)\le 0\right]<\infty
\end{equation}
and
\begin{equation}\label{persiV4}
\liminf_{N\rightarrow +\infty}\frac{c_N\sqrt{\log c_N}}{\sigma_{c_N}}
    P\left[\max_{k=1,...,c_N}V(k)\le 0\right]>0.
\end{equation}
We conclude the proof of \eqref{eqLemmaFrankbis2r} by gathering \eqref{EqProbaTaloglogN}, \eqref{persiV1},  \eqref{persiV2}, \eqref{persiV3} and \eqref{persiV4}.
\end{proof}

%
%
%
%
%

%%%%%%%%%%%%%%%%%%%%%%%%%%%%%%%%%%%%%%%%%%%%%%%%%%%%%%%%%%%%%%%%%%%%%%%%%%%%%%%%%

\subsection{Random walk in a bad environment}\label{SubSectionBadEnv}

%%%%%%%%%%%%%%%%%%%%%%%%%%%%%%%%%%%%%%%%%%%%%%%%%%%%%%%%%%%%%%%%%%%%%%%%%%%%%%%%%%

Let $a\in (1,+\infty)$, $q>1$ and $\varepsilon \in (0,1)$.
Define, as before,  $b_N:=\sup\{k\ :\ \sigma_k\le (\log N)(\log\log N)^{-\frac q{2}}\}$.
When $V=B_H$, $b_N=b(\log N)$ with the notation of the proof of Theorem
\ref{frank} before \eqref{eqn:firstdist1}.
We define, for $N\geq 3$, a set $\mathcal{B}_N$ of ``bad environments'' (that happen with large probability) as follows
$$
    \mathcal{B}_N
:=
    %\bigcap_{i=1}^2 \mathcal{B}_N^{(i)},
    \mathcal{B}_N^{(1)}\cap \mathcal{B}_N^{(2)},
$$
where
\begin{eqnarray*}
    \mathcal{B}_N^{(1)}
& := &
    \left\{  T\left(a \log \log N\right)  \leq b_N \leq T\left(-\frac{(1-\varepsilon)}{2}\log N\right)\right\},
%\\
%    \mathcal{B}_N^{(2)}
%& := &
%    \left\{ \big| V(-1) \big| \leq \frac 1 2 \log \log N\right\},
%\\
%    \mathcal{B}_N^{(3)}
%& := &
%    \left\{ \big| V(-1) - V(-2) \big| \leq \frac 1 2 \log \log N\right\}.
\\
    \mathcal{B}_N^{(2)}
& := &
    \bigcap_{i=-1}^{b_N}\left\{ \big| V(i) - V(i-1) \big| \leq \frac 1 2 \log \log N\right\}.
\end{eqnarray*}

%\begin{figure}
% \includegraphics[width=0.9\textwidth]{bad_env}
%\caption{Sketch of a bad environment $\omega\in\mathcal{B}_N$.
%\label{bad_env}}
%\end{figure}

We first study the behavior of a random walk in a bad environment, and show that its quenched probability of persistence is small.

\medskip

\begin{lem}\label{uniform}
Let $a\in (1,+\infty)$. We have for large enough $N$,
$$
    \forall \omega\in\mathcal{B}_N,
\qquad
    P_{\omega} \Big[ \min_{k=1,\ldots, N} S_k > -1 \Big]
=
    P_{\omega}\Big[ \tau(-1)>N\Big]
\leq
    \frac{2}{(\log N)^{a-1}}.
$$
\end{lem}

\smallskip

\begin{proof}
Let $a>1$, $N\geq 3$, $\omega\in\mathcal{B}_N$ and $\alpha:=T\left(a \log \log N\right)$.  Let us decompose
\begin{eqnarray}
    P_{\omega}\big[ \tau(-1)>N\big]
& \leq &
    P_{\omega} \big[\tau(-1)\wedge \tau(\alpha+1) >N ; \ \tau(-1)<\tau(\alpha +1)\big]
    +P_{\omega}\big[\tau(-1)>\tau(\alpha+1)\big]
\nonumber\\
&=: &
    P_1(\omega) + P_2(\omega).
\label{eqDefP1P2}
\end{eqnarray}
From \eqref{eqProbaAtteinte},
using the definition of $\alpha$ and the fact that $\omega\in \mathcal{B}_N^{(2)}$, we get
%the probability $P_2(\omega)$ is upper bounded by
\begin{equation}\label{InegP2}
    P_2(\omega)
\leq
    e^{V(-1)}\bigg(\displaystyle\sum_{k=-1}^{\alpha}e^{V(k)}\bigg)^{-1}
\leq
    \frac{e^{V(-1)}}{e^{V(\alpha)}}\leq (\log N)^{1-a}.
\end{equation}
Note that
$
    1+e^{X_\ell}
=
    e^0+e^{V(\ell)-V(\ell-1)}
\leq
    2 \exp\big[\max_{-2\leq j\leq k\leq \alpha}(V(k)-V(j))\big]
$
for every $-1\leq \ell \leq \alpha$.
Moreover,
$V(k)\leq a\log\log N +|V(\alpha)-V(\alpha-1)|\leq 2a \log\log N$
for all $0\leq k \leq \alpha$,
and
$
    \alpha
<
    T[-(1-\varepsilon)(\log N)/2]
$.
Hence from \eqref{InegEsperanceZeitouni} and Markov's inequality,
%the probability $P_1(\omega)$ is upper bounded by
\begin{eqnarray*}
    P_1(\omega)
\leq
    \frac{1}{N} E_{\omega} [ \tau(-1)\wedge \tau(\alpha+1)]
& \leq &
    \frac{2}{N} (\alpha +2)^2 \exp\Big\{ 2 \max_{-2\leq \ell\leq k\leq \alpha} (V(k)-V(\ell))\Big\}\\
& \leq &
    \frac{2}{N} (\alpha +2)^2 \exp\big\{ (1-\varepsilon) \log N + 4a \log \log N\big\}
\\
& = &
    \frac{2}{N} (\alpha +2)^2 N^{(1-\varepsilon)} (\log N)^{4a}\\*
& \leq &
    c N^{-{\varepsilon}} (\log N)^{4a+\frac 2H}(L(\log N))^2,
\end{eqnarray*}
uniformly for large $N$, where we used $\omega\in \mathcal{B}_N$ in the second step
and $\omega\in \mathcal{B}_N^{(1)}$ and $b_N=(\log N)^{\frac 1 H}L(\log N)$ (see Lemma \ref{PERSIbis}) in the last step.
This together with \eqref{eqDefP1P2} and \eqref{InegP2} proves the lemma.
\end{proof}

\medskip

We now treat the probability of the bad environment.

\begin{lem}\label{BN}
There exists $c>0$ such that, for $N$ large enough,
\begin{equation}\label{eqLemmaBNCaseV}
    P\left[\mathcal{B}_N^{c}\right]
\leq
   \frac{\sigma_{b_N}}{b_N}(\log\log N)^c.
\end{equation}
\end{lem}
\begin{proof}
It is enough to upper bound each probability $P\left[(\mathcal{B}_N^{(i)})^c\right]$, $i=1,2$.
Since $(V(i)-V(i-1))$ follows a standard Gaussian distribution for every $i\in\ZZ$, we have for large $N$,
\begin{equation}\label{InegProbaBN2}
    P\left[(\mathcal{B}_N^{(2)})^c\right]
%\leq
%    (b_N+2)P\big[|V(1)|\geq (1/2)\log\log N\big]
\leq
    2(b_N+2)
    %(\log N)^{\frac{2}{H}}
    \exp\big(-(\log \log N)^2/8\big)
\leq
    (\log N)^{-\frac{1}{H}}.
\end{equation}
%the probabilities of the events $(\mathcal{B}_N^{(2)})^c$ and
%$(\mathcal{B}_N^{(3)})^c$ are upper bounded by $2 \exp(-(\log \log N)^2/8)$.
The upper bound of $P\left[(\mathcal B_N^{(1)})^c\right]$ comes from Lemma \ref{PERSIbis}
in the general case, that is when $V$ is not necessarily $B_H$, which proves \eqref{eqLemmaBNCaseV}.
\end{proof}

The proof of the upper bound of Theorem \ref{persist} directly follows from Lemma \ref{uniform} and \ref{BN}. Indeed,
\begin{eqnarray*}
    \PP\big[ \tau(-1) > N \big]
& = &
    \int_{\mathcal{B}_N} P_{\omega} \big[ \tau(-1) > N \big] {\rm d} P(\omega)
    +
    \int_{\mathcal{B}_N^c} P_{\omega} \big[ \tau(-1) > N \big] {\rm d} P(\omega)
\\*
&\leq &
    \frac{2}{(\log N)^{a-1}}
    +
    P(\mathcal{B}_N^c ),
\end{eqnarray*}
where we used Lemma \ref{uniform}. Let us choose $a > 1 + \frac{1-H}{H}$. The upper bound of Theorem \ref{persist} follows from Lemma \ref{BN}.
\section{Proof of the lower bound in Theorem \ref{persist}}\label{lowerMain}
\subsection{Good environments}
Let $\gamma:=T(1)$ and $\varepsilon>0$.
For $N\geq 3$, we consider a set $\mathcal G_N$ of (rare) good environments:
$$
    {\mathcal G}_N
:=
    \mathcal G_N^{(1)}\cap \mathcal{G}_N^{(2)} \cap \mathcal{G}_N^{(3)}\cap \mathcal{G}_N^{(4)},
$$
with
\begin{eqnarray*}
    \mathcal{G}_N^{(1)}
& := &
    \Big\{\beta_N:=T(-2 \log N)<\gamma\Big\},
\\
    \mathcal{G}_N^{(2)}
& := &
    \{\gamma< (\log N)^{\frac {1+\varepsilon}H}\},
\\
    \mathcal{G}_N^{(3)}
& := &
    \left\{\sup_{|k|\le (\log N)^{\frac {1+\varepsilon}H}}|X_k|\le \sqrt{\frac{(4-2H)(1+\varepsilon)}{H}\log\log N} \right\},
\\
    \mathcal{G}_N^{(4)}
& := &
    \left\{ \Big(\sum_{k=0}^{\beta_N-1}e^{V(k)}\Big)^{-1} \geq f(N)\right\},
\end{eqnarray*}
where $f(N):= \frac{1}{\kappa (\log \log N)^2}$ with
$\kappa:=5(2/H)^2$.
%$\kappa>0$ appropriately chosen later (see just after \eqref{InegfofN2a}).
%and $\widetilde{\mathcal{G}}_N: = \mathcal{G}_N \cap \mathcal{G}_N^{(4)} .$
If $\omega \in  \mathcal{G}_N$, we say that it is a ``good environment''.

\begin{figure}
 \includegraphics[width=0.9\textwidth]{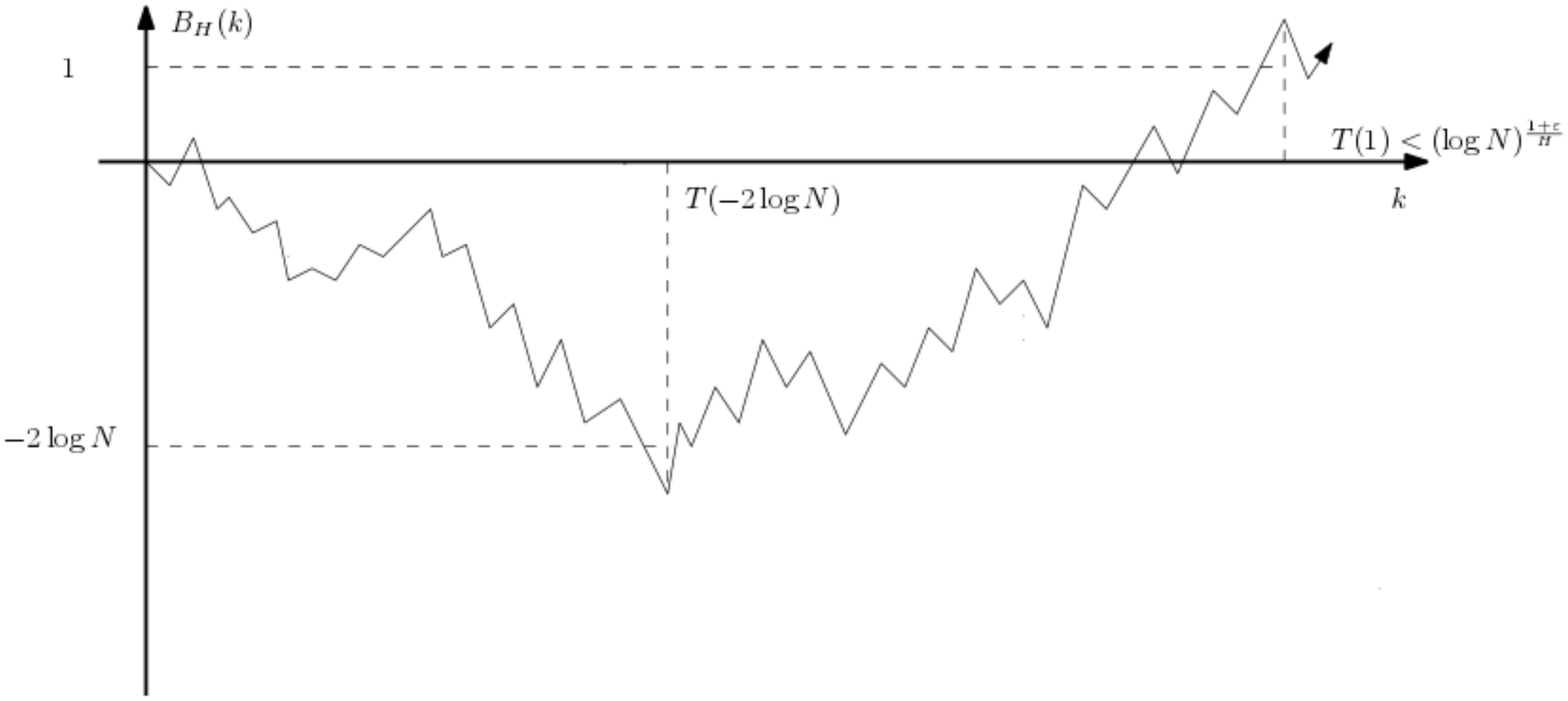}
\caption{Sketch of a good environment $\omega\in\mathcal{G}_N$.
\label{good_env}}
\end{figure}

\subsection{Random walks in good environments}
We shall prove in Lemma \ref{lem:ppviagoodenv} that the persistence probability is directly related to the probability of good environments. So we just need to give a lower bound for $P\big(\mathcal{G}_N\big)$.

\begin{lem} \label{lem:ppviagoodenv}
With the notation for $\mathcal{G}_N$ defined above, we have for $N$ large enough
\begin{equation}
    \mathbb P[\tau(-1)>N]
\geq     e^{-c \sqrt{\log\log N}} P\big(\mathcal{G}_N\big).
\label{InegLowerBoundPersistenceGN}
\end{equation}
\end{lem}

\begin{proof}
Since $\beta_N>0$, for every $\omega\in{\mathcal G_N}$, we have by \eqref{eqProbaAtteinte},
\begin{equation}\label{EEEEEE1}
    P_\omega\left[\tau(\beta_N)<\tau(-1)\right]
=
    \bigg(1+\sum_{k=0}^{\beta_N-1}e^{V(k)-V(-1)}\bigg)^{-1}.
\end{equation}
Moreover, under $P_\omega^{\beta_N}$, the number of excursions of $(S_k)_{k\ge 0}$ from $\beta_N$ to $\beta_N$
without visiting neither $-1$ nor $\gamma$ is geometric with parameter $p$ given by
$$
    p
:=
    P_\omega^{\beta_N}\left[\tau(-1)\wedge\tau(\gamma)<\tau(\beta_N)\right].
$$
Using the fact that $-1<\beta_N<\gamma$, we observe that once more by \eqref{eqProbaAtteinte},
\begin{eqnarray*}
    p
& = &
    \frac 1{1+e^{X_{\beta_N}}}\left(P_\omega^{\beta_N+1}\left[\tau(\gamma)<\tau(\beta_N)\right]+
          e^{X_{\beta_N}}P_\omega^{\beta_N-1}\left[\tau(-1)<\tau(\beta_N)\right]\right)
\\
& = &
    \frac 1{1+e^{X_{\beta_N}}}\left(  \frac{e^{V(\beta_N)}}{\sum_{k=\beta_N}^{\gamma-1}e^{V(k)}}+
          e^{X_{\beta_N}}\frac{e^{V(\beta_N-1)}}{\sum_{k=-1}^{\beta_N-1}e^{V(k)}}\right)
\\
&\le&
    \frac 1{1+e^{X_{\beta_N}}} e^{V(\beta_N)}\left(  \frac{1}{e^{V(\gamma-1)}}+
         \frac{1}{e^{V(-1)}}\right)
\\
&\le&
    \frac 1{1+e^{X_{\beta_N}}} e^{-2\log N}\left(  \frac{1}{e^{1-X_{\gamma}}}+
         \frac{1}{e^{-X_0}}\right)
\\
&\le&
    2e^{\sqrt{\frac{(4-2H)(1+\varepsilon)}{H}\log\log N}} e^{-2\log N}
\\
&\le&
    %c N^{-2}(\log N)^{\frac{(4-2H)(1+\varepsilon)}H},
    (\log N)N^{-2},
\end{eqnarray*}
for every $\omega\in{\mathcal G}_N$, for $N$ large enough.
Since $\tau(-1)\wedge\tau(\gamma)$ is larger than this number of excursions, we conclude that
there exists $N_0\ge 0$ such that for every $N\ge N_0$
and every $\omega\in{\mathcal G}_N$, the following estimate holds
\begin{equation}\label{EEEEEE2}
    P_\omega^{\beta_N}\big[\tau(-1)\wedge\tau(\gamma)> N\big]
\ge
    (1-p)^{N+1}
\ge
    \frac 1 2.
\end{equation}
Hence, due to \eqref{EEEEEE1} and \eqref{EEEEEE2}, there exists $N_0\ge 3$ such that for every $N\ge N_0$
and every $\omega\in{\mathcal G}_N$, we have by the strong Markov property applied at time $\tau(\beta_N)$,
$$
    P_\omega\left[\tau(-1)>N\right]
\ge
    P_\omega\left[ \tau(\beta_N)<\tau(-1)\right]
    P_\omega^{\beta_N}\left[ \tau(-1)\wedge\tau(\gamma)>N\right]
\ge
    %\frac {1}2\frac{e^{V(-1)}}{\sum_{k=-1}^{\beta_N-1}e^{V(k)}} .
    \frac {1}2 \bigg(1+\sum_{k=0}^{\beta_N-1}e^{V(k)-V(-1)}\bigg)^{-1}.
$$
Hence, for every integer $N\ge N_0$,
\begin{eqnarray}
    \mathbb P[\tau(-1)>N]&\ge&\int_{{\mathcal G}_N}P_\omega\left[\tau(-1)>N\right]\, dP(\omega)
\nonumber\\
&\ge&
    \frac {1}2\int_{{\mathcal G}_N}\Big(1+\sum_{k=0}^{\beta_N-1}e^{V(k)-V(-1)}\Big)^{-1}\, dP(\omega).
\label{EqProbaAnnealedRetourMoins1}
\end{eqnarray}
Note that on the set $\mathcal G_N$, $|V(-1)|\leq \sqrt{\frac{(4-2H)(1+\varepsilon)}{H}\log\log N}$,
so
$
    1+\sum_{k=0}^{\beta_N-1}e^{V(k)-V(-1)}
\leq
    \big(1+e^{|V(-1)|}\big)
    \sum_{k=0}^{\beta_N-1}e^{V(k)}
\leq
    2 e^{\sqrt{\frac{(4-2H)(1+\varepsilon)}{H}\log\log N}}
    \sum_{k=0}^{\beta_N-1}e^{V(k)}
$.
Hence, thanks to \eqref{EqProbaAnnealedRetourMoins1}, there exists $c>0$ such that for large $N$,
\begin{eqnarray*}
    \mathbb P[\tau(-1)>N]
& \geq &
    \frac{1}{4} e^{-\sqrt{\frac{(4-2H)(1+\varepsilon)}{H}\log\log N}}\int_{{\mathcal G}_N}\Big(\sum_{k=0}^{\beta_N-1}e^{V(k)}\Big)^{-1}\, dP(\omega)
\nonumber\\
& \geq &
    \frac{1}{4} e^{-\sqrt{\frac{(4-2H)(1+\varepsilon)}{H}\log\log N}}f(N)P\big(\mathcal{G}_N\big)
    \nonumber\\
& \geq &
    e^{-c \sqrt{\log\log N}} P\big(\mathcal{G}_N\big),
% \label{InegLowerBoundPersistenceGN}
\end{eqnarray*}
as stated.
\end{proof}

\subsection{Probability of good environments}
This subsection is devoted to the proof of the following lemma:

\begin{lem}\label{LemmaProbaGN}
There exists $\tilde L_0$ a slowly varying function at infinity such that for large $N$,
$$
     P\big(\mathcal{G}_N\big)
\geq
    (\log N)^{-\left(\frac{1-H}{H}\right)} \tilde L_0(\log N).
$$
Moreover  if $V=B_H$, there exists $c>0$ such that
$
     P\big(\mathcal{G}_N\big)
\geq
    (\log N)^{-\left(\frac{1-H}{H}\right)} (\log\log N)^{-c}
$
for large $N$.
\end{lem}

The proof of Lemma \ref{LemmaProbaGN} relies on the following technical result.

\begin{lem} \label{lem:lemma5}
There exists $\tilde L_0$ a slowly varying function at infinity such that
$$
    P\Big[\mathcal{G}_N^{(1)}\cap \mathcal{G}_N^{(4)}\Big]
\geq
    (\log N)^{-\frac{1-H}{H}} \tilde L_0(\log N).
$$
If moreover $V=B_H$, then there is a $c>0$ such that for large enough $N$
$$
    P\Big[\mathcal{G}_N^{(1)}\cap \mathcal{G}_N^{(4)}\Big]
\geq
    (\log N)^{-\frac{1-H}{H}} (\log\log N)^{-c}.
$$
\end{lem}

\smallskip

\begin{proof}[Proof of Lemma \ref{LemmaProbaGN}]
Note that, by Theorem~11 of \cite{AGPP}, large enough $N$
\begin{equation}\label{eqProbaGN2}
    %P\left[\gamma\ge (\log N)^{\frac {1+\varepsilon}H}\right]
    P\Big[\big(\mathcal{G}_N^{(2)}\big)^c\Big]
\le
    (\log N)^{-\left(\frac{(1-H)(1+\varepsilon)}{H}\right)}\ell\left((\log N)^{\frac {1+\varepsilon}H}\right) % (\log \log N)^{c'}.
\end{equation}
Moreover,  for $N$ large enough, we have since $X_k\sim\mathcal{N}(0,1)$ for every $k\in\ZZ$,
\begin{eqnarray}
    %P\left[\sup_{|k|\le (\log N)^{\frac {1+\varepsilon}H}}|X_k|> \sqrt{\frac{(4-2H)(1+\varepsilon)}{H}\log\log N} \right]
    P\Big[\big(\mathcal{G}_N^{(3)}\big)^c\Big]
& \le &
    3(\log N)^{\frac {1+\varepsilon}H}P\left[|X_0|> \sqrt{\frac{(4-2H)(1+\varepsilon)}{H}\log\log N} \right]
\nonumber\\
& \le  &
    3(\log N)^{\frac {1+\varepsilon}H} e^{-\frac{(2-H)(1+\varepsilon)}{H}\log\log N}
\nonumber\\
& = &
   3(\log N)^{-\frac {(1-H)(1+\varepsilon)}H}.
\label{eqProbaGN3}
\end{eqnarray}
% we conclude that
% \begin{equation}\label{VERYGOOD}
% P\Big[\mathcal{G}_N\Big] \geq (\log N)^{-\left(\frac{1-H}{H}\right)}(\log \log N)^{-c}.
% \end{equation}
Due to Lemma \ref{lem:lemma5}, for large $N$,
$$
     P\big(\mathcal{G}_N\big)
\geq
    P\big[\mathcal{G}_N^{(1)}\cap {\mathcal{G}}_N^{(4)}\big]
    -P\big[\big(\mathcal{G}_N^{(2)}\big)^c\big]
    -P\big[\big(\mathcal{G}_N^{(3)}\big)^c\big]
\geq
    (\log N)^{-\left(\frac{1-H}{H}\right)} \tilde L_0(\log N),
$$
since the probability of the sets
$\big(\mathcal{G}_N^{(2)}\big)^c$ and $\big(\mathcal{G}_N^{(3)}\big)^c$
are of lower order by \eqref{eqProbaGN2} and \eqref{eqProbaGN3}.
Similarly,
$
     P\big(\mathcal{G}_N\big)
\geq
    (\log N)^{-\frac{1-H}{H}} (\log\log N)^{-c}
$
for large $N$ if $V=B_H$.
\end{proof}

\smallskip

\begin{proof}[Proof of Lemma \ref{lem:lemma5}]
{\it Step 1:}
%Let $d:=d(N):=\max\{k\ :\ \sigma_k\le 2\log N(\log\log N)^{\frac q{2}}\}$, with $q>1$
Let $K:=K_N:=\min\{k\in\NN\, :\, \sigma^2_k\ge 33 (2\log N)^2\}$
and $L:=L_N:=\lfloor (\log\log N)^{\frac q{2H}}\rfloor$, with $q>2H/(4-4H)$ and $q>2H$.
Moreover for every $\varepsilon>0$,
$
    \sigma_{K}^2
\le
    2(\sigma_{K-1}^2+\sigma_1^2)
\leq
    (33\times 2^3+\varepsilon)(\log N)^2
$ for large $N$.
Due to Karamata's characterization of regularly varying
functions (see Karamata \cite[Thm III]{Karamata} or Bingham, Goldie and Teugels \cite[Thm 1.3.1]{BGT}), for a fixed  $u>0$,
for $N$ large enough, we have
\begin{equation}\label{ratio}
  \forall j\ge 1,\quad     j^{2H-u} \le \frac{\sigma^2_{jK}}{\sigma^2_{K}}=j^{2H}\frac{\ell(jK)}{\ell(K)}
          \le  j^{2H+u}
\end{equation}
(we can take $u=0$ if $\ell=1$).
Consequently for such $u$,
$8(\log N)(\log\log N)^{\frac{q(H-u)}{2H}}\le \sigma_{LK}\le 17 (\log N)(\log\log N)^{\frac{q(H+u)}{2H}}$
for large $N$.
Set $d:=LK$.
%Since
%%and $K:=K_N:=\min\{k\ :\ E[(V(2k)-2V(k))^2]\ge 64(2\log N)^2\}$. Since
%$\sigma^2_d\ge \lfloor d/K\rfloor \sigma^2_K$, $L:=L_N:=\lfloor d/K\rfloor\le (\log\log %N)^{ q}$.
Then,
\begin{eqnarray}
    P\big[ \mathcal{G}_N^{(1)} \cap \mathcal{G}_N^{(4)}\big]
& =&
    P\bigg[ T(-2\log N) < T(1), \frac{1}{\sum_{k=0}^{T(-2\log N)-1} e^{V(k)}} \geq f(N) \bigg]
\notag \\
%&\geq&
%    P[ T(-2\log N) < d < T(1), \frac{1}{\sum_{k=0}^{T(-2\log N)} e^{V(k) }} \geq f(N) ]
%\notag \\
& \geq &
    P\bigg[ T(-2\log N) < d < T(1), \frac{1}{\sum_{k=0}^{d} e^{V(k) }} \geq f(N) \bigg]
\notag \\
&=&
    P\bigg[ T(1)>d, \frac{1}{\sum_{k=0}^{d} e^{V(k)}} \geq f(N) \bigg]
\notag \\
&& \qquad\qquad
    - P\bigg[ T(-2\log N) \geq d,  T(1)>d, \frac{1}{\sum_{k=0}^{d} e^{V(k)}} \geq f(N) \bigg].
    \label{eqn:firstst}
\end{eqnarray}
For the last term in (\ref{eqn:firstst}), we will apply Li and Shao \cite[Thm. 4.4]{lishao} with  $\xi_i:=V(iK)-V((i-1)K)$,
$X(t)=V(dt)$ for $t$ multiple of $1/d$, $a= 1/L$ and with $\varepsilon= 2 \log N$, and note that $X(ia)=V(id/L)=V(iK)$.
Observe that $E[\xi_i^2]=E[(V(K))^2]=\sigma^2_{K}$.
Hence $L^{-1}\sum_{i=2}^{L}{E[\xi_i^2]}\ge 32(2\log N)^2$ and, due to \cite[Thm. 4.4]{lishao}, we obtain
\begin{multline*}
    P\bigg[ T(-2\log N) \geq d, T(1)>d, \frac{1}{\sum_{k=0}^{d} e^{V(k)}} \geq f(N) \bigg]
\leq
    P\Big[ \max_{k=1,\ldots,d} |V(k)| \leq 2\log N\Big]\\
\leq \exp\left(-\frac{(2\log N)^4}{16 L^{-2} \sum_{i,j=2}^{L}(E[\xi_i\xi_j])^2}\right)
\leq \exp\left(-\frac{(2\log N)^4}{32 L^{-1} \sum_{j=2}^{L}(E[\xi_2\xi_j])^2}\right),
\end{multline*}
where we used
$
    \sum_{j=2}^L  (E[\xi_i\xi_j])^2
\leq
    \sum_{k=-L+4}^L  (E[\xi_2\xi_k])^2
\leq
    2\sum_{j=2}^L  (E[\xi_2\xi_j])^2
$
by stationarity
for every $i\in\{2,\dots,L\}$
 in the last inequality,
Moreover, note that by Cauchy–-Schwarz inequality,
$
    \sum_{j=2}^5(E[\xi_2\xi_j])^2
\leq
    \sum_{j=2}^5 (E[\xi_2^2])(E[\xi_j^2])
=
    4 \sigma^4_K
=
    O((\log N)^4)
$
and that, for every $j\ge 6$,
$$
E[\xi_2\xi_j]=[\sigma_{ (j-1)K}^2-2\sigma^2_{ (j-2)K}+\sigma_{ (j-3)K}^2]/2.
$$
Recall that
$$
    \sigma_{jK}^2=\sum_{n,m=1}^{jK}r(n-m)=jK+2\sum_{m=1}^{jK}(jK-m)r(m).
$$
Hence only the $r(m)$'s with $m> (j-3)K$ contribute to $E[\xi_2\xi_j]$ and, for every $u>0$,
there exists $C_0>0$, such that
$$
\forall j\ge 6,\quad |E[\xi_2\xi_j]|\le  (2K)^2 \sup_{m\,:\,|m- (j-2)K|\le K}r(m)
\le C_0 j^{2H-2+u}\sigma^2_K,
$$
if $N$ is large enough, since $m^2r(m)=O(\sigma^2_m)$
(see \cite[Prop 1.5.8, Prop 1.5.9a]{BGT} as in \eqref{formulaVarVn}) and due to \eqref{ratio}.
%\footnote{Indeed, noting that $\sigma^2_m=m\, r(0)+2\sum_{k=1}^m(m-k)\, r(k)$, we can choose $a\in(0,1)$ such that
%$\liminf_{m\rightarrow +\infty}\frac{\sigma^2_m}{2\sum_{k=1}^{\lfloor am\rfloor} (m-k)\, r(k)}>1/2$.}
Since $\sigma_K^2=O((\log N)^2)$, it comes that, if $N$ is large enough
\begin{eqnarray}
    P\bigg[ T(-2\log N) \geq d, T(1)>d, \frac{1}{\sum_{k=0}^{d} e^{V(k)}} \geq f(N) \bigg]
& \leq &
    \exp\left(-c \min(L,L^{4-4H-u})\right)
\nonumber\\
& \leq &
    (\log N)^{-\frac{1-H}{H}-1}
\label{InegfofN1}
\end{eqnarray}
for some $c>0$, where we take $u>0$ such that $(4-4H-u)q/(2H)>1$ and since $q/(2H)>1$.

For the first term in (\ref{eqn:firstst}), observe that
\begin{eqnarray}
&&
    P\bigg[ T(1)>d, \frac{1}{\sum_{k=0}^{d} e^{V(k)}} \geq f(N) \bigg]
\nonumber\\
& \geq &
    P\Big[V(k)< 1, k=1,\ldots,\lfloor\log d\rfloor^2; V(k) \leq - \log d, k=\lfloor\log d\rfloor^2+1,\ldots, d\Big],
\label{InegfofN2a}
\end{eqnarray}
because if $V$ satisfies the conditions inside the previous probability
(because $\kappa=5(2/H)^2$ in the definition of $f(N)$, since for large $N$,
$d=(\log N)^{\frac 1H}\tilde L(\log N)\leq (\log N)^{\frac 2H}$
for some $\tilde L$ slowly varying at infinity):
$$
    \sum_{k=0}^{d} e^{V(k)}
=
    1 + \sum_{k=1}^d e^{V(k)}
\leq
    1+(\log d)^2 \, e^1 + d e^{-\log d}
\leq
    5(\log d)^2
\leq
    \kappa(\log\log N)^2
=
    f(N)^{-1}.
$$

\noindent
{\it Step 2:}
In order to show the lemma, in view of  \eqref{eqn:firstst},  \eqref{InegfofN1}  and \eqref{InegfofN2a},
it remains to study
\[
    P\big[V(k)< 1, k=1,\ldots,\lfloor\log d\rfloor^2; V(k)\leq - \log d, k=\lfloor\log d\rfloor^2+1,\ldots, d\big].
\label{eqn:secondstar}
\]
For this purpose, first observe that by Slepian's lemma,
\begin{eqnarray}
& & P\big[V(k)< 1, k=1,\ldots,\lfloor\log d\rfloor^2; V(k)\leq - \log d, k=\lfloor\log d\rfloor^2+1,\ldots, d\big]
\notag\\
& \geq &
    P\big[V(k)< 1, k=1,\ldots,\lfloor\log d\rfloor^2\big] \cdot P\big[V(k)\leq - \log d, k=\lfloor\log d\rfloor^2+1,\ldots, d\big].
    \label{eqn:rerslep}
\end{eqnarray}
Let us look at the first term in the right hand side of \eqref{eqn:rerslep}.
Applying the maximal inequality in Proposition~2.2 in Khoshnevisan and Lewis \cite{KL}
as in the start of the proof of our Lemma \ref{PERSIbis}, we  can write
\begin{eqnarray}
    P\Big[\max_{k=1,\ldots,\lfloor\log d\rfloor^2} V(k) < 1\Big]
&=&
    1-P\Big[\max_{k=1,\ldots,\lfloor\log d\rfloor^2} V(k) \geq 1\Big]
\nonumber\\
& \geq &
    1-2 P\Big[V \big( \lfloor\log d\rfloor^2 \big) \geq 1\Big]
\nonumber\\
& = & 
    P\Big[|V(1)| < \lfloor\log d\rfloor^{-2H}(\ell(\lfloor\log d\rfloor^2))^{-\frac12}\Big]\nonumber\\
&\geq&
    c \ \lfloor\log d\rfloor^{-2H}(\ell(\lfloor\log d\rfloor^2))^{-\frac12}, \label{eqn:rmod2}
\end{eqnarray}
since $V(1)\sim\mathcal N(0,1)$.
Let us now consider the second term in (\ref{eqn:rerslep}):
\begin{eqnarray}
&&
    P\big[V(k)\leq - \log d, k=\lfloor\log d\rfloor^2+1,\ldots, d\big]
\notag \\
&\geq &
    P\big[ V( \lfloor\log d\rfloor^2 )  \leq - \log d; V(k)-V(\lfloor\log d\rfloor^2)\leq 0, k=\lfloor\log d\rfloor^2+1,\ldots, d\big]
\notag \\
&\geq &
    P\big[ V( \lfloor\log d\rfloor^2 ) \leq - \log d\big]
    \cdot
    P\big[V(k)-V(\lfloor\log d\rfloor^2)\leq 0, k=\lfloor\log d\rfloor^2+1,\ldots, d\big],
\label{eqn:rmod3}
\end{eqnarray}
where the last step follows from Slepian's lemma and we use that $r(i)\ge 0$ for all $i\in\ZZ$.
The first term in \eqref{eqn:rmod3} equals
\begin{equation}
    P\big[ V( \lfloor\log d\rfloor^2 ) \leq - \log d\big]
=
    P\left[  V(1) \leq - \frac{\log d}{\sigma_{\lfloor\log d\rfloor^2}}\right]
\geq
    P[ V(1) \leq - 2]
\geq
    \mbox{const.},
\label{eqn:rmod4}
\end{equation}
for $N$ large enough, since $\sigma_T^2\ge \sum_{i=1}^T\mathbb E[X_i^2]=T$.
The second term in (\ref{eqn:rmod3}) is bounded below by Theorem~11 in \cite{AGPP}:
\begin{eqnarray*}
&& P[V(k)-V(\lfloor\log d\rfloor^2)\leq 0, k=\lfloor\log d\rfloor^2+1,\ldots, d]
\\
&=& P[V(k)\leq 0, k=1,\ldots, d-\lfloor\log d\rfloor^2]
\\
&\geq & c\frac{d^{-(1-H)}\sqrt{\ell(d)}}{\sqrt{\log d}},
\end{eqnarray*}
for some $c>0$.
Putting this together with \eqref{InegfofN2a}, \eqref{eqn:rerslep}, \eqref{eqn:rmod2},
\eqref{eqn:rmod3}, and \eqref{eqn:rmod4},
we obtain
\[
P\bigg[ T(1)>d, \frac{1}{\sum_{k=0}^{d} e^{V(k)}} \geq f(N) \bigg]
\nonumber\\
 \geq c'\frac{d^{-(1-H)}\sqrt{\ell(d)}}{(\log d)^{2H+\frac 12}\sqrt{\ell((\log d)^2)}}
 \, ,
\]
for some $c'>0$ ;
which, combined with \eqref{eqn:firstst}, \eqref{InegfofN1} and with the
definition of $d$, gives the result.
\end{proof}

\smallskip
Finally, putting \eqref{InegLowerBoundPersistenceGN} from Lemma~\ref{lem:ppviagoodenv} and Lemma \ref{LemmaProbaGN} together proves the
lower bound in Theorem~\ref{persist}.

\section{Proof of Theorem \ref{main1}}\label{proofThm2}
We start by stating the following lemma, which will be helpful to analyze asymptotic quantities coming from the hitting time formula (\ref{eqProbaAtteinte}). However, we believe that this lemma may be of independent interest (cf.\ the continuous time analogs in \cite{Molchan1999,Aurzada}).

% --------- taken from A-GP, first paper 'time reversible' ---------------------

\begin{lem}\label{Mol} Let $Z=(Z_n)_{n\in\NN}$ be a stochastic process with
\begin{equation} \label{ass1}
    \lim_{T\to+\infty} \frac{1}{T^{H}\ell(T)} \EE\left[ \sup_{  t\in[0,1]} Z_{\lfloor t T\rfloor} \right] = \kappa,
\end{equation}
for some $H\in(0,1)$, $\kappa\in(0,\infty)$, and with $\ell$ being a slowly varying function at infinity. Further assume that $Z$ is time-reversible in the sense that for any $T\in\NN$, the vectors $(Z_{T-k}-Z_T)_{k=0,\ldots,T}$ and $(Z_k)_{k=0,\ldots,T}$ have the same law.
Then,
$$
    \limsup_{x\rightarrow +\infty} \frac{x^{1-H}}{ \ell(x)} \EE\bigg[ \Big(\displaystyle\sum_{l=0}^{\lfloor x\rfloor} e^{Z_l}\Big)^{-1} \bigg]
\leq
    \kappa H
$$
and
$$
    \liminf_{x\rightarrow +\infty} \frac{x^{1-H}}{\ell(x)} \EE\bigg[ \Big(\displaystyle\sum_{l=1}^{\lfloor x\rfloor} e^{Z_l}\Big)^{-1} \bigg]
\geq
    \kappa H.
$$
%Moreover,  assume that
%\begin{equation} \label{ass2}
%\EE[e^{-2 Z_n}] = o(n^{H-1} \ell(n)),
%\end{equation}
%then
%\begin{equation} \label{equality}
%\lim_{x\rightarrow +\infty} \frac{x^{1-H}}{ \ell(x)} \EE\Big[ \Big(\displaystyle\sum_{l=0}^{[x]} e^{Z_l}\Big)^{-1} \Big]=
%\lim_{x\rightarrow +\infty} \frac{x^{1-H}}{\ell(x)} \EE\Big[ \Big(\displaystyle\sum_{l=1}^{[x]} e^{Z_l}\Big)^{-1} \Big] = \kappa H.
%\end{equation}
\end{lem}

Note the difference in the summation $l=0,\ldots$ vs.\ $l=1,\ldots$, which complicates the use of this lemma.
%In particular, if one can show the second assertion with summation in $l=0,1,\ldots$, one would be able to reduce the error term in the lower bound from $e^{-c\sqrt{\log T}}$ to $(\log T)^{-c}$ in (\ref{eqMk}), (\ref{eqn:thmd2}), and (\ref{eqn:thmlrd}); cf.\ (\ref{eqn:summation}).

In fact, it suffices to have the two terms in question bounded from above and below, respectively. For this purpose, one could replace  (\ref{ass1}) by the weaker assumption that
$$
\frac{1}{T^{H}\ell(T)} \EE\left[ \sup_{  t\in[0,1]} Z_{\lfloor t T\rfloor} \right]
$$
is bounded away from zero and infinity for large $T$.

\begin{proof}
Let us define for every $T\in [1,+\infty),$
$$ \Psi(T):= \EE\left[ \log\bigg(\sum_{k=0}^{\lfloor T\rfloor-1} e^{Z_k} + (T- \lfloor T\rfloor)e^{Z_{\lfloor T\rfloor} } \bigg)\right].$$
We clearly have
$$
    \EE\Big[ \sup_{t\in [0,1]} Z_{\lfloor t(\lfloor T\rfloor -1)\rfloor}\Big]
\leq \Psi(T) \leq \EE\Big[ \sup_{t\in [0,1]} Z_{\lfloor tT\rfloor}\Big]  +\log (T+ 1).$$
From assumption (\ref{ass1}), it follows that $\Psi(T)\sim \kappa T^H \ell(T) $ as $T\to+\infty$.

By Fubini's theorem we have for any $u\in (1,+\infty)$,
\begin{eqnarray*}
    \Psi(u)
&=&
    \EE\bigg[ \log \Big( \sum_{l=0}^{\lfloor u\rfloor -1} e^{Z_l} + (u-\lfloor u\rfloor) e^{Z_{\lfloor u\rfloor} } \Big) \bigg] \\
&=&
    \int_1^u \Psi'(x) \dd x,
\end{eqnarray*}
where for every $x\geq 1$,
\begin{eqnarray*}
    \Psi'(x)
& = &
    \sum_{k=1}^{\infty} \EE\left[ \frac{e^{Z_k}}{\sum_{l=0}^{k-1} e^{Z_l} +(x-k) e^{Z_k}} \right] {\bf 1}_{[k,k+1)}(x)
\\
& = &
    \sum_{k=1}^{\infty} \EE\left[ \frac{1}{\sum_{l=1}^{k} e^{Z_{k-l}-Z_k} +(x-k)} \right] {\bf 1}_{[k,k+1)}(x).
\end{eqnarray*}
Using time reversibility,
$$
    \Psi'(x) =\sum_{k=1}^{\infty} \EE\left[ \frac{1}{\sum_{l=1}^{k} e^{Z_l} +(x-k)} \right] {\bf 1}_{[k,k+1)}(x).
$$
Let $0< a<b <+\infty$. Then,  for $x$ large enough,
\begin{eqnarray}\label{diff}
    \Psi( \lfloor bx\rfloor+1) -\Psi( \lfloor ax\rfloor)
&=&
    \int_{\lfloor ax\rfloor}^{\lfloor bx\rfloor+1} \Psi'(u) \dd u\nonumber \\
&=&
    \sum_{k=\lfloor ax\rfloor}^{\lfloor bx\rfloor}  \int_k^{k+1} \EE\left[ \frac{1}{\sum_{l=1}^{k} e^{Z_l} +(u-k)} \right] \, \dd u.
\end{eqnarray}
\noindent{\bf Estimation of the limsup:} Estimating the last quantity from below, we get the inequality:
\begin{eqnarray*}
    \Psi \big( \lfloor bx\rfloor+1\big) -\Psi\big( \lfloor ax\rfloor\big)
&\geq &
    \big(\lfloor bx\rfloor  -\lfloor ax\rfloor +1\big) \EE\bigg[ \Big(\sum_{l=0}^{\lfloor bx\rfloor} e^{Z_l} \Big)^{-1}\bigg].
\end{eqnarray*}
Therefore,
$$
    \frac{x^{1-H}}{\ell(x)} \EE\bigg[ \Big(\sum_{l=0}^{\lfloor bx\rfloor} e^{Z_l}\Big)^{-1} \bigg]
\leq
    \frac{\Psi\big( \lfloor bx\rfloor+1\big) -\Psi\big( \lfloor ax\rfloor\big)}{ x^H \ell(x)}
    \frac{x}{\big(\lfloor bx\rfloor  -\lfloor ax\rfloor +1 \big)}.
$$
Since
\begin{eqnarray}
\frac{\Psi( \lfloor bx\rfloor +1) -\Psi( \lfloor ax\rfloor)}{ x^H \ell(x)}
&=&
    \frac{\Psi( \lfloor bx\rfloor +1)}{ x^H \ell(x)} - \frac{\Psi( \lfloor ax\rfloor )}{ x^H \ell(x)} \notag \\
&=&
    \frac{\Psi( \lfloor bx\rfloor+1)}{ (\lfloor bx\rfloor+1)^H \ell(\lfloor bx\rfloor+1) }
    \frac{(\lfloor bx\rfloor+1)^H \ell(\lfloor bx\rfloor+1)}{x^H \ell(x)} \notag
\\
&&-
    \frac{\Psi( \lfloor ax\rfloor)}{ \lfloor ax\rfloor^H \ell(\lfloor ax\rfloor)}
    \frac{\lfloor ax\rfloor^H \ell(\lfloor ax\rfloor)}{x^H \ell(x)} \notag \\
&\rightarrow &
    \kappa (b^H -a^H), \label{eqn:remark2ofrefeee}
\end{eqnarray}
as $x\rightarrow +\infty$, we obtain
$$
    \limsup_{x\rightarrow +\infty} \frac{x^{1-H}}{\ell(x)} \EE\bigg[ \Big(\displaystyle\sum_{l=0}^{\lfloor bx\rfloor} e^{Z_l}\Big)^{-1} \bigg]
\leq
    \frac{\kappa \big(b^H -a^H\big) }{b-a}.
$$
Now taking $b=1$ and $a\uparrow 1$, we get
$$
    \limsup_{x\rightarrow +\infty} \frac{x^{1-H}}{\ell(x)} \EE\bigg[ \Big(\displaystyle\sum_{l=0}^{\lfloor x\rfloor} e^{Z_l}\Big)^{-1} \bigg]
\leq
    \kappa H.
$$
{\bf Estimation of the liminf:}  First, from (\ref{diff}),  we have the inequality
\begin{eqnarray*}
    \Psi\big( \lfloor bx\rfloor+1\big) -\Psi\big( \lfloor ax\rfloor\big)
&\leq &
    \big(\lfloor bx\rfloor  -\lfloor ax\rfloor +1 \big) \EE\bigg[ \Big(\sum_{l=1}^{\lfloor ax\rfloor} e^{Z_l} \Big)^{-1}\bigg].
\end{eqnarray*}
Therefore,
$$
    \frac{x^{1-H}}{\ell(x)}
    \EE\bigg[ \Big(\sum_{l=1}^{\lfloor ax\rfloor} e^{Z_l}\Big)^{-1} \bigg]
\geq
    \frac{\Psi\big( \lfloor bx\rfloor+1\big) -\Psi\big( \lfloor ax\rfloor\big)}{ x^H \ell(x)}
    \frac{x}{\big(\lfloor bx\rfloor  -\lfloor ax\rfloor +1\big)}.
$$
By the same argument as in (\ref{eqn:remark2ofrefeee}), we obtain
% \begin{eqnarray*}
% \frac{\Psi( [bx]+1) -\Psi( [ax])}{ x^H \ell(x)} &=& \frac{\Psi( [bx]+1)}{ x^H \ell(x)} - \frac{\Psi( [ax])}{ x^H \ell(x)} \\
% &=& \frac{\Psi( [bx]+1)}{ ([bx]+1)^H  \ell([bx]+1)}\frac{([bx]+1)^H \ell([bx]+1)}{x^H \ell(x)}
% \\
% && - \frac{\Psi( [ax])}{ [ax]^H \ell([ax])}\frac{[ax]^H \ell([ax])}{x^H \ell(x)}\\
% &\rightarrow & \kappa (b^H -a^H),
% \end{eqnarray*}
% as $x\rightarrow +\infty$, we obtain
% $$\liminf_{x\rightarrow +\infty} \frac{x^{1-H}}{\ell(x)} \EE\Big[ (\displaystyle\sum_{l=1}^{[ax]} e^{Z_l})^{-1} \Big]\geq \frac{\kappa (b^H -a^H) }{(b-a)}.$$
% $a=1$ and $b\downarrow 1$, we get
$$
    \liminf_{x\rightarrow +\infty} \frac{x^{1-H}}{\ell(x)}
    \EE\bigg[ \Big(\displaystyle\sum_{l=1}^{\lfloor x\rfloor} e^{Z_l}\Big)^{-1} \bigg]
\geq
    \kappa H.
$$
This ends the proof of Lemma \ref{Mol}.
%\noindent{\bf Proof of (\ref{equality}):}
%Let $M_n:= \sum_{k=1}^n e^{Z_k}$. By remarking that $1+M_n= \sum_{k=0}^n e^{Z_k}$, (\ref{equality}) directly follows from the inequality
%\begin{eqnarray*}
%0\leq \EE\left[\frac{1}{M_n}\right] -\EE\left[\frac{1}{1+M_n}\right] &\leq& \EE\left[M_n^{-2}\right]\\
%&\leq &\EE\left[ e^{-2 Z_n}\right] =o(n^{H-1} \ell(n))
%\end{eqnarray*}
%from assumption (\ref{ass2}).
\end{proof}
\begin{proof}[Proof of Theorem \ref{main1}]
From the definition of $\mathcal T$, the correspondence \eqref{eqDefBPRE} between the BPCGE $(Z_n)_{n\geq 0}$ and the RWCGE $(S_n)_{n\geq 0}$, and formula (\ref{eqProbaAtteinte}), we have
\begin{eqnarray} \label{tailT}
\PP[\mathcal T>N] &=& \PP[ Z_N>0]\nonumber\\
&=&
    \PP[ \tau(N) < \tau(-1)]\nonumber\\
&=&
    E [ P_{\omega} [\tau(N) < \tau(-1) ] ]\nonumber\\
&=&
    E \bigg[ e^{V(-1)}   \Big(\sum_{k=-1}^{N-1} e^{V(k)} \Big)^{-1} \bigg]\nonumber\\
&=&
    E \bigg[ \Big(\sum_{k=0}^{N} e^{V(k)} \Big)^{-1} \bigg],
\label{eqProofUpperBoundTh2}
\end{eqnarray}
using that the increments of $V$ are stationary.
We now explain how the upper bound of Theorem \ref{main1} can be deduced from our Lemma~\ref{Mol}.
Since our $V$ satisfies the hypotheses of that lemma (see the proof of Theorem 11 in \cite{AGPP}) we have
$$
    E \bigg[ \Big(\sum_{k=0}^{N} e^{V(k)} \Big)^{-1} \bigg]
\leq
    c \frac{\sqrt{\ell(N)}}{N^{1-H}}
$$
for large $N$, with $\ell=1$ if $V=B_H$. This, combined with \eqref{eqProofUpperBoundTh2} proves the upper bound of Theorem~\ref{main1}.

Moreover, set $\beta_N:=\log N$ and choose $a_N$
such that $\sigma_{a_N}\sim_{N\to+\infty} \beta_N$.
Set $\phi(k,N)=0$ if $k< a_N$ and $\phi(k,N)=-\beta_N$ otherwise.
Due to Slepian's lemma (using the non-negative correlations of the $X_i$ and thus $V$),
\begin{eqnarray*}
&&
    E \bigg[ \Big(\sum_{k=0}^{N} e^{V(k)} \Big)^{-1} \bigg]
\\
&\ge&
    \left(\sum_{k=0}^Ne^{\phi(k,N)}\right)^{-1}
    P\Big[\forall k=1,...,N,\  V(k)\le \phi(k,N)\Big]
\\
&\ge&
    (a_N+Ne^{-\beta_N})^{-1}P\left[\max_{k=0,...,a_N}V(k)\le 0\right]
    P\Big[V(a_N)\le -\beta_N\Big]
\\
&&
    \qquad\qquad\qquad\qquad\qquad\qquad\qquad\qquad\qquad\qquad
    \times
    P\left[\max_{k=a_N+1,...,N}(V(k)-V(a_N))< 0 \right]
\\
&\ge&
    (a_N+Ne^{-\beta_N})^{-1}P\left[\max_{k=0,...,a_N}V(k)\le 0\right]
    P\Big[\sigma_{a_N}V(1)\le -\beta_N\Big]P\left[
          \max_{k=1,...,N-a_N}V(k)< 0 \right]
\\
&\ge& K (a_N+1)^{-1}\frac{\sigma_{a_N}}{a_N\sqrt{\log a_N}}
   \frac{\sigma_N}{N\sqrt{\log N}},
\end{eqnarray*}
due to \cite[Theorem 11]{AGPP}.
This proves the lower bound of Theorem~\ref{main1}.
\end{proof}

    \medskip
    
{\noindent\bf Acknowledgments.} We are grateful to Nina Gantert for stimulating discussions.

\end{document}